\documentclass[11pt,a4paper,reqno]{amsart}
\usepackage[latin1]{inputenc}
\usepackage[T1]{fontenc}
\usepackage{lmodern}
 \usepackage[english]{babel,minitoc}
\usepackage{hyperref}
\usepackage{amsfonts}
\usepackage{amsmath}

\numberwithin{equation}{section}\theoremstyle{plain}
\newtheorem{theorem}{Theorem}[section]

\newtheorem{lemma}[theorem]{Lemma}

\newtheorem{definition}[theorem]{Definition}

\newenvironment{Proof}[1][Proof]{\noindent\textbf{#1.} }{\ \rule{0.5em}{0.5em}}

\address{\begin{center}{\small Department of Mathematics and Computer Sciences, Faculty of Sciences,\\
Equipe d'Analyse Harmonique et Probabilit\'{e}s, University Moulay Isma\"{\i}l,\\
BP 11201 Zitoune, Meknes, Morocco}
\end{center}}

\begin{document}

\title[Helgason-Gabor Fourier Transform and  Uncertainty Principles]
{ Helgason-Gabor Fourier Transform and  Uncertainty Principles }

\author[M.El kassimi M.Boujeddaine and S.Fahlaoui]{ M. El kassimi, M.Boujeddaine and S. Fahlaoui }

\address{Mohammed El kassimi} \email{m.elkassimi@edu.umi.ac.ma}

\address{Mustapha Boujeddaine}\email{boujeddainemustapha@gmail.com}

\address{Sa\"{\i}d Fahlaoui}  \email{s.fahlaoui@fs.umi.ac.ma}

\maketitle
\begin{abstract}
Windowing a Fourier transform is a useful tool, which gives us the similarity between the signal and time frequency signal, and it allows to get sense when/where ceratin frequencies occur in the input signal, this method is introduced by Dennis Gabor. In this paper, we generalize the classical Gabor-Fourier transform(GFT) to the  Riemannian symmetric space called the Helgason Gabor Fourier transform (HGFT). We continue with proving several important properties of HGFT, like  the reconstruction formula, the Plancherel formula, and Parseval formula. Finally we establish some local uncertainty principle such as Benedicks-type uncertainty principle.
\end{abstract}

\section*{Introduction}
The Fourier transform has been a useful tool for analyzing frequency properties of a  signal, but this transform still insufficient to represent and compute location information for  a given signal.  To solve this problem, in \cite{Gabor1} Gabor  formulated a fundamental method by multiplying   the function to be transformed  by a Gaussian function. This transform becomes a powerful method for determining the sinusoidal frequency and phase content of signal local sections considering  its changes over time. In addition, it used for filtering and modifying the signal in the limited region.\\ In classical case, the Gabor transform is given by (\cite{Gabor1})
$$ \mathcal{G}\{f\}(b,\omega)=\int_{-\infty}^{+\infty}f(t)e^{-\pi(t-\tau)^2}e^{-i\omega t}dt$$
In the general case, we take the windowed function $\varphi$ as square integral function. The Gabor Fourier transform has other names used in the literature, like as short-time Fourier transform and windowed Fourier transform. 
 Motivated by this concept, in this paper we study a generalization of the classical Gabor Fourier transform to the Riemannian Symmetric spaces see \cite{Helgason1,Mohanty}, which we call the Helgason Gabor Fourier transform(HGFT).Then, we derive important harmonic analysis properties of HGFT.\\
 This paper is organized as follows, in the first section we remind some results about the classical Helgason-Fourier transform, in the second we define the HGFT, and we establish for it a several harmonic analysis properties, such as the inversion formula, Plancherel and Parseval formulas, in the last one we demonstrate some local uncertainty principles for HGFT like Benedick's theorem.

\section{Helgason Transform}
\subsection{Helgason Transform}
In this section we describe the necessary preliminaries regarding semi-simple Lie groups and
harmonic analysis on associated Riemannian symmetric spaces.

If $X$ is a Riemannian symmetric space of noncompact type then $X$ can be viewed as a
quotient space $G/K$ where $G$ is a connected, noncompact, semi-simple Lie group with finite
center and $K$ a maximal compact subgroup of $G$.

Let $G =NAK$ be an Iwasawa decomposition of $G$
and let $\mathfrak{a}$ be the Lie algebra of $A$. Denoting by $M$ the
centralizer of $A$ in $K$ and putting $B = K/M$. By writing
$g=n.exp A(g).u$,$g \in G$,   where $u\in K, A(g)\in \mathfrak{a}$,
$n\in N$, and  for $x=gK \in X$ and $b=kM \in B = K/M$, we write
$A(x,b)=A(k^{-1}g)$. Let $dx$ be a $G$-invariant measure on $X$,
and let $db$ and $dk$ be the respective normed $K$-invariant
measures on $B$ and $K$.

Let $o = {K}$ be the origin in $X$ and denote the action of $G$ on
$X$ by $(g,x)\longmapsto gx$ for $g\in G,\, x\in X$. The Lie
algebras of $G$ and $K$ are respectively denoted by $\mathfrak{g}$ and
$\mathfrak{k}$.

We denote by $\mathcal{C}_{c}^{\infty}(X)$ the set of infinity
differentiable compactly-supported functions on $X$. Let $dg$ be
the element of the Haar measure on G.

We assume that the Haar measure on $G$ is normed, so that
\begin{equation}\label{equation00}
\int_{X} f(x)\mbox{d}x = \int_{G} f(go)\mbox{d}g, \quad f\in
\mathcal{C}_{c}^{\infty}(X)
\end{equation}

Let $\mathfrak{a}^{\ast}$ be the real dual of   $\mathfrak{a}$ and
$\mathfrak{a}^{\ast}_{\mathbb{C}}$ be its complexification ; the
finite Weyl group $W$ acts on $\mathfrak{a}^{\ast}$. Suppose that
$\Sigma$ is the set of bounded roots $(\Sigma \subset
\mathfrak{a}^{\ast})$, $\Sigma^{+}$ is the set of positive bounded
roots, and $\mathfrak{a}^{+}$ is the positive Weyl chamber so that

\begin{displaymath}
\mathfrak{a}^{+} = \{h \in \mathfrak{a} : \alpha(h) > 0\, \mbox{for}\quad
\alpha \in \Sigma^{+}\}.
\end{displaymath}
Denote by $\rho$ the half-sum of the positive bounded roots
(counted with their multiplicities) ; then $\rho\in
\mathfrak{a}^{\ast}$. Let $\langle , \rangle$ be the Killing form on the Lie algebra $\mathfrak{a}$ . For $\lambda\in \mathfrak{a}^{\ast}$, let
$A_{\lambda}$ be the vector in $\mathfrak{a}$ such that $\lambda (A)=
\langle A_{\lambda}, A\rangle$ for all $A\in \mathfrak{a}$. Given
$\lambda,\mu \in \mathfrak{a}^{\ast}$, we set  $\langle \lambda,\mu
\rangle : = \langle A_{\lambda}, A_{\mu}\rangle$. The
correspondence $\lambda\mapsto A_{\lambda}$ enables us to identify
$\mathfrak{a}^{\ast}$ with $\mathfrak{a}$. Using this identification, we
can translate the action of the Weyl group $W$ to $\mathfrak{a}$. Let
$$\mathfrak{a}^{\ast}_{+} = \{\lambda \in \mathfrak{a}^{\ast}:
A_{\lambda}\in \mathfrak{a}^{\ast}\}$$

The Helgason Fourier transform is a powerful tool in harmonic analysis on noncompact
Riemannian symmetric spaces $G/K$ (\cite{Mohanty}). This transform associates to any smooth
compactly supported right $K$-invariant function $f$ on $G$.

For integrable functions $f$ on $\mathcal{C}_{c}^{\infty}(X), b\in K$ and $\lambda\in
\mathfrak{a}^{\ast}$, Helgason-Fourier transform is defined as in (\cite{Helgason1})by:

\begin{equation}
\widehat{f}(\lambda,b)= \int_{X}f(x)
e^{(-i\lambda+\rho)(A(x,b))}\mbox{d}x,\quad \lambda \in
\mathfrak{a}^{\ast}, b\in B=K/M
\end{equation}

We will assume throughout this paper $X$ is of rank 1, and hence  dim
$\mathfrak{a}^{\ast}$ = 1.\,
In this case we identify $\mathfrak{a}^{\ast}_{ \mathbb{C}}$ with
$\mathbb{C}$ by identifying $\lambda_{\rho}$ with $\lambda
\in\mathbb{C}$. Under this identification, $\mathfrak{a}^{\ast} =
\mathbb{R}$ by means of the correspondence $\lambda \mapsto
\lambda \alpha, \lambda \in \mathbb{R}$.

We norm the measure on $X$ and we conclude this section with the
following properties, due to Helgason.

The original function $f\in\mathcal{C}_{c}^{\infty}(X)$ can then be reconstructed from $\widehat{f}$ by means of the inversion formula

\begin{equation}
f(x)= \frac{1}{|W|}\int_{\mathfrak{a}^{\ast}\times
B}\widehat{f}(\lambda,b)
e^{(i\lambda+\rho)(A(x,b))}|c(\lambda)|^{-2}\mbox{d}\lambda
\mbox{d}b,
\end{equation}
where $|W|$ is the order of the Weyl group of $G/K$, $d\lambda$ is the element of the Euclidean measure on $\mathfrak{a}^{\ast}$ and $c(\lambda)$
is the Harish-Chandra function.\\


We also state the Plancherel formula for the Fourier transform:

\begin{theorem}
The Fourier transform defined on $C_{c}(X)$ by (1) extends to an
isometry of $L^{2}(X)$ onto $L^{2}(\mathfrak{a}^{\ast}\times B)$ (with
the measure $|c(\lambda)|^{-2} \mbox{d}\lambda \mbox{d}b$ on
$\mathfrak{a}^{\ast}\times B$). Moreover,
\begin{equation}
\int_{X}f_{1}(x)\overline {f_{2}(x)}\mbox{d}x =
\frac{1}{|W|}\int_{\mathfrak{a}^{\ast}\times
B}\widehat{f_1}(\lambda,b)\overline {\widehat{f_2}(\lambda,b)}
e^{(i\lambda+\rho)(A(x,b))}|c(\lambda)|^{-2}\mbox{d}\lambda
\mbox{d}b,
\end{equation}
for all $f_{1}, f_{2} \in L^{2}(X)$.
\end{theorem}

\begin{Proof}
See \cite[Theorem 2, page 227]{Helgason2}.
\end{Proof}

It follows from the above arguments that for $f \in L^{2}(X)$, we
have

\begin{equation}
\int_{X}|f(x)|^{2}\mbox{d}x =
\frac{1}{|W|}\int_{\mathfrak{a}^{\ast}\times
B}|\widehat{f}(\lambda,b)|^{2} \mbox{d}\mu(\lambda)\mbox{d}b =
\int_{\mathfrak{a}^{\ast}_{+}\times B}|\widehat{f}(\lambda,b)|^{2}
\mbox{d}\mu(\lambda)\mbox{d}b,
\end{equation}
where $\mbox{d}\mu(\lambda):= |c(\lambda)|^{-2} \mbox{d}\lambda$. \\




Given $h\in G$. For a function $f\in \mathcal{C}_{0}(X)$, the translation operator $T_{h}$ is given by the formula
$$
(T_{h}f)(x) := \displaystyle \int_{K} f(gkho)\mbox{d}k.
$$


\vspace{0,3cm}

We remind that a function $\varphi \in  \mathcal{C}_{0}(X)$ is
called a spherical function if $\varphi$ is $K$-invariant,
$\varphi(o)=1$, and for each $D \in \mathcal{C}_{c}^{\infty}(X)$,
there exists $\lambda_{D}\in \mathcal{C}$ such that $D
\varphi=\lambda_{D}\varphi$.\\

We now list down some well known properties of the elementary spherical functions  on $X$ based on the Harish-Chandras result \cite[Chapitre 4, Theorem 4.3]{Helgason1}.

First, we give the following lemma proved in \cite[Lemma 3]{Helgason1}.
\begin{lemma}
For $f\in L^{2}(X)$, we have
\begin{displaymath}
\widehat{(T_{h}(f))}(\lambda,b)=
\varphi_{\lambda}(h).\widehat{f}(\lambda,b),\quad h\in G
\end{displaymath}
where $\widehat{f}(\lambda,b)$ is the Fourier
transform of $f$
\end{lemma}


\section{Gabor-Helgason transform}
The classical Gabor transform of a function $f\in L^2(\mathbf{R})$ cannot possess a support of finite Lebesgue measure. In \cite{Wilczok}
the author showed that the portion of this transform lying outside some set $M$ of finite
Lebesgue measure cannot be arbitrarily small, either. For sufficiently small $M$,
this can be seen immediately by estimating the Hilbert-Schmidt norm of a
suitably defined operator. In this section, we try to give some new harmonic analysis results  related to Gabor transform in the case of Riemannian symmetric space $X$.

We define first the Gabor-Helgason transform by:
$$\mathcal{G}_{\varphi}\{f\}(\lambda,b,h)=\int_{X}f(x)\overline{T_{h^{-1}}\varphi(x)}e^{(-i\lambda+\rho)(A(x,b))}dx$$
with $\lambda \in a^{\ast}$, $b\in B=K/M$

\subsection{Inversion formula}

Before to give the reconstruction formula for the HGFT, we need the following lemma, which proves that, the translation  $T_h$ is an isometric operator for the norm of the space $L^2$.

\begin{lemma}\label{lemma-translation}
  For every function $f\in L^2(X)$ and  $h\in G$:
We have
  $$\|T_{h}f\|^2_{L^{2}(X)}=\|f\|^2_{L^{2}(X)}$$
\end{lemma}

\begin{proof}
  Applying the relation \eqref{equation00}, we get,
  \begin{eqnarray*}
   \|T_{h}f\|^2_{L^{2}(X)} &=& \int_{X}|T_{h}f(x)|^2dx\\
   &=&\int_{G}|T_{h}f(go)|^2dg \\
   &=& \int_{G}T_{h}f(go)\overline{T_{h}f(go)}dg\\
   &=& \int_{G} \int_{K}f(gkho)dk \int_{K}\overline{f(gk^{'}ho)}dk^{'}dg\\
   &=&  \int_{K} \int_{K}(\int_{G}f(gkho)\overline{f(gk^{'}ho)} dg)dkdk^{'}\\
   &=& \int_{K} \int_{K}(\int_{G}f(go)\overline{f(go)} dg)dkdk^{'}\\
   &=&\int_{G}|f(go)|^2dg\\
    &=&\|f\|^2_{L^{2}(X)}
  \end{eqnarray*}

\end{proof}

\begin{theorem}
Let $\varphi \in L^2(X)$ be a window function. Then every function $f\in L^2(X)$, can be reconstructed by
$$f(x)=\frac{1}{|W|\|\varphi\|^2_2}\int_{X}\int_{a^{\ast}\times B}\mathcal{G}^{\varphi}(\lambda,b,h)e^{(i\lambda+\rho)(A(x,b))}T_{h^{-1}}\varphi(x)|c(\lambda)|^{-2}d\lambda dbdh$$
\end{theorem}

\begin{proof}
  We can obtain the inversion formula by using the fact that:
  $$\mathcal{G}_{\varphi}\{f\}(\lambda,b,h)=(\widehat{f\overline{T_{h^{-1}}\varphi}})(\lambda,b)$$
  So,

 \begin{equation}\label{equa1}
  f(x)\overline{T_{h^{-1}}\varphi(x)}=\frac{1}{|W|}\int_{a^{\ast}\times B} \mathcal{G}_{\varphi}\{f\}(\lambda,b,h)e^{(i\lambda+\rho)(A(x,b))}|c(\lambda)|^{-2}d\lambda db,
 \end{equation}

 We multiply the both sides of \eqref{equa1} by $T_{h^{-1}}\varphi$, we obtain
 \begin{equation}\label{equa2}
    f(x)|T_{h^{-1}}\varphi(x)|^{2}=\frac{1}{|W|}\int_{a^{\ast}\times B} \mathcal{G}_{\varphi}\{f\}(\lambda,b,h)e^{(i\lambda+\rho)(A(x,b))}|c(\lambda)|^{-2}d\lambda db T_{h^{-1}}\varphi(x)
 \end{equation}

 on integer the inequality \eqref{equa2} with respect the measure $dh$, w get
 $$
 f(x)\int_{G}|T_{h^{-1}}\varphi(x)|^{2}dh=\frac{1}{|W|}\int_{G}\int_{a^{\ast}\times B} \mathcal{G}_{\varphi}\{f\}(\lambda,b,h)e^{(i\lambda+\rho)(A(x,b))}|c(\lambda)|^{-2} T_{h^{-1}}\varphi(x)d\lambda dbdh
$$
 by using the first lemma \ref{lemma-translation}, we obtain,
\begin{equation}\label{equa3}
 f(x)\|\varphi\|^2_{L^{2}(X)}=\frac{1}{|W|}\int_{G}\int_{a^{\ast}\times B} \mathcal{G}_{\varphi}\{f\}(\lambda,b,h)e^{(i\lambda+\rho)(A(x,b))}T_{h^{-1}}\varphi(x)|c(\lambda)|^{-2}d\lambda dbdh
\end{equation}

Now, simplifying  both sides of \ref{equa3} by $ \|\varphi\|^2_{L^{2}(X)}$, we get our result.
 $$
  f(x)=\frac{1}{|W|\|\varphi\|^2_{L^{2}(X)}}\int_{G}\int_{a^{\ast}\times B} \mathcal{G}_{\varphi}\{f\}(\lambda,b,h)e^{(i\lambda+\rho)(A(x,b))} T_{h^{-1}}\varphi(x)|c(\lambda)|^{-2}d\lambda dbdh
 $$

\end{proof}

\begin{theorem}[Plancherel formula]\label{Plancherel}
  For $f\in L^{2}(X)$ and $\varphi\in L^{2}(X) $ a windowed function, we have
  $$\|\mathcal{G}_{\varphi}\{f\}\|^{2}_{L^{2}(a^{\ast}\times B\times G)}=\|f\|^2_{L^{2}(X)}\|\varphi\|^2_{L^{2}(X)}$$
\end{theorem}
\begin{proof}
  We have,
  \begin{eqnarray*}
    \|\mathcal{G}_{\varphi}\{f\}\|^{2}_{L^{2}(a^{\ast}\times B\times G)} &=& \|\widehat{f \overline{T_{h^{-1}}\varphi}}\|^2_{L^{2}(a^{\ast}\times B\times G)} \\
     &=& \|f \overline{T_{h^{-1}}\varphi}\|^2_{L^{2}(X\times G)}  \\
     &=& \int_{G}\int_{X}f(x) T_{h^{-1}}\varphi(x)\overline{f(x) T_{h^{-1}}\varphi(x)}dxdh \\
     &=& \int_{G}\int_{X}f(x) T_{h^{-1}}\varphi(x)\overline{T_{h^{-1}}\varphi(x)}~\overline{f(x)}dxdh \\
     &=& \int_{G}\int_{X}|f(x)|^2 |T_{h^{-1}}\varphi(x)|^2 dxdh
     \end{eqnarray*}
 using the equation \eqref{equation00}, lemma \ref{lemma-translation} and the Fubini's theorem, we have,
      \begin{eqnarray*}
     &=& \int_{G}\int_{G}|f(go)|^2 \int_{K}\varphi(gkh^{-1}o)dk \int_{K}\overline{\varphi(gk^{'}h^{-1}o)}dk^{'}dhdg \\
     &=&  \int_{G}|f(go)|^2 \int_{G}\int_{K} \int_{K}\varphi(gkh^{-1}o)\overline{\varphi(gk^{'}h^{-1}o)}dkdk^{'}dhdg
     \end{eqnarray*}
 using the invariance of the Haar measure $dg$ by K, we get
      \begin{eqnarray*}
     &=&  \int_{G}|f(go)|^2 \int_{G}\varphi(h^{-1}o)\int_{K} \int_{K}\overline{\varphi(h^{-1}o)}dkdk^{'}dhdg  \\
     &=&  \int_{G}|f(go)|^2 \int_{K}dk \int_{K}dk^{'}\int_{G}|\varphi(h^{-1}o)|^2 dhdg  \\
     &=&\int_{X}|f(x)|^2dx\int_{X}|\varphi(y)|^2dy\\
     &=&\|f\|^2_{L^{2}(X)}\|\varphi\|^2_{L^{2}(X)}
  \end{eqnarray*}
\end{proof}
\begin{theorem}[Parseval's identity ]
Let $\varphi\in L^{2}(X)$ be a window function and $f,g\in L^{2}(X)$ arbitrary. Then we have

$$\int_{G\times a^\ast\times B }\mathcal{G}_{\varphi}\{f\}(\lambda,b,h)\overline{G_{\varphi}g}(\lambda,b,h)|c(\lambda)|^{-2}d\lambda db dh=\|\varphi\|^2_{L^{2}(X)}\int_{X}f(x)\overline{g(x)}dx$$

\end{theorem}
\begin{proof}
we have by the lemma \ref{lemma-translation}
$$\int_{G\times a^\ast\times B }G_{\varphi}\{f\}(\lambda,b,h)\overline{G_{\varphi}g}(\lambda,b,h)|c(\lambda)|^{-2}d\lambda db dh$$
  \begin{eqnarray*}
     &=& \int_{G\times a^\ast\times B}\widehat{(f(.)T_{h^{-1}}\varphi(.))}(\lambda,b,h)\overline{\widehat{(g(.)T_{h^{-1}}\varphi(.))}}(\lambda,b,h)|c(\lambda)|^{-2}d\lambda db dh \\
     &=&\int_{G}\int_{X}f(x)T_{h^{-1}}\varphi(x)\overline{g(x)T_{h^{-1}}\varphi(x)}dx  dh \\
     &=& \int_{G}\int_{X}f(x)T_{h^{-1}}\varphi(x)\overline{T_{h^{-1}}\varphi(x)}\,\overline{g(x)}dx  dh \\
     &=& \int_{G}\int_{X}f(x)\overline{g(x)}|T_{h^{-1}}\varphi(x)|^2 dx dh \\
     &=&  \int_{X}f(x)\overline{g(x)}dx\int_{G}|T_{h^{-1}}\varphi(y)|^2 dh \\
      \end{eqnarray*}
 Such as the proof of the Plancherel's theorem \ref{Plancherel}, Applying  the equation \eqref{equation00}, lemma \ref{lemma-translation} and the Fubini's theorem, we get,
      \begin{eqnarray*}
     &=& \int_{G}\int_{G}f(go)\overline{g(go)} \int_{K}\varphi(gkh^{-1}o)dk \int_{K}\overline{\varphi(gk^{'}h^{-1}o)}dk^{'}dhdg \\
     &=&  \int_{G}f(go)\overline{g(go)} \int_{G}\int_{K} \int_{K}\varphi(gkh^{-1}o)\overline{\varphi(gk^{'}h^{-1}o)}dkdk^{'}dhdg
     \end{eqnarray*}
 using the invariance of the Haar measure $dg$ by K, we obtain
      \begin{eqnarray*}
     &=&  \int_{G}f(go)\overline{g(go)} \int_{G}\varphi(h^{-1}o)\int_{K} \int_{K}\overline{\varphi(h^{-1}o)}dkdk^{'}dhdg  \\
     &=&  \int_{G}f(go)\overline{g(go)} \int_{K}dk \int_{K}dk^{'}\int_{G}|\varphi(h^{-1}o)|^2 dhdg  \\
     &=&\int_{X}f(x)\overline{g(x)}dx\int_{X}|\varphi(y)|^2dy\\
      &=& \|\varphi\|^2_{L^{2}(X)}\int_{X}f(x)\overline{g(x)}dx
  \end{eqnarray*}

\end{proof}

We shall now discuss the validity of some uncertainty principles in the   case of Gabor-Helgason transform.

\section{Uncertainty principle}

In quantum physics, the uncertainty principles  state that, we cannot give simultaneously  the position and moment time of particle with high precision. The formulation mathematics of this concept is that, the function and its Fourier transform cannot both be sharply localized. Many formulations are given, the first one is proved by Heinseberg in 1927 \cite{Heisenberg}, after, many authors give some generations, such as, Hardy's theorem \cite{Hardy}, Morgan's theorem \cite{Morgan}. Years after, the locally uncertainty principles arise, those theorems asset that, when the uncertainty of the momentum is small, the probability of being localized at any point is very small \cite{Benedicks,Donoho,Price}.\\

 Our first result will be the following local uncertainty principle,

 \begin{lemma}
  Let $\varphi, f\in L^2(X)$,  we have
\begin{equation}\label{young-inequality}
\|\mathcal{G}_{\varphi}\{f\}\{f\}(\omega,y)\|_{L^{\infty}(a^{\ast}\times B\times X)} \leq \|f\|_{L^2(X)}\|\varphi\|_{L^2(X)}
\end{equation}
\end{lemma}
 \begin{proof}
   We have
\begin{eqnarray*}
  |\mathcal{G}_{\varphi}\{f\}\{f\}(\lambda,b,h))| &=& |\int_{X}f(x)\overline{T_{h^{-1}}\varphi(x)}e^{(-i\lambda+\rho)(A(x,b))}dx| \\
      &\leq& \int_{X}|f(x)||\overline{T_{h^{-1}}\varphi(x)}|dx
\end{eqnarray*}
Using H\"{o}lder inequality we get our result
\begin{equation*}
\|\mathcal{G}_{\varphi}\{f\}\{f\}(\lambda,b,h)\|_{L^{\infty}(a^{\ast}\times B\times X)} \leq \|f\|_{L^2(X)}\|\varphi\|_{L^2(X)}
\end{equation*}
\end{proof}

\begin{theorem}
  Let $\varphi$ a windowed function and let $\Sigma$ a subset of $a^{\ast}\times B\times X$ such that $0<m(\Sigma)<+\infty$, for all $f\in L^2(X)$ we have,
  \begin{equation}\label{theo-concentration}
  \displaystyle\|f\|_{L^2(X)} \|\varphi\|_{L^{2}(X)}\leq \frac{1}{\sqrt{1-m(\Sigma)^2}}\|\mathcal{G}_{\varphi}\{f\}\chi_{\Sigma^{c}}\|_{L^2(a^{\ast}\times B\times X)}
  \end{equation}
\end{theorem}

\begin{proof}
  For every $f\in L^{2}(X)$; we have
  \begin{equation}\label{concentration12}
   \|\mathcal{G}_{\varphi}\{f\}\|^{2}_{L^2(X)}=\|\mathcal{G}_{\varphi}\{f\}\chi_{\Sigma}\|^{2}_{L^2(a^{\ast}\times B\times X)}+\|\mathcal{G}_{\varphi}\{f\}\chi_{\Sigma^{c}}\|^{2}_{L^2(a^{\ast}\times B\times X)}
  \end{equation}

 Applying the \eqref{young-inequality} and  the Plancherel formula \ref{Plancherel}, we get
  \begin{eqnarray}
    \|\mathcal{G}_{\varphi}\{f\}\chi_{\Sigma}\|^{2}_{L^2(a^{\ast}\times B\times X)} &\leq& m(\Sigma)^2 \|\mathcal{G}_{\varphi}\{f\}\|^2_{L^{\infty}(a^{\ast}\times B\times X)}\\
     &\leq& m(\Sigma)^2 \|\varphi\|^2_{L^{2}(X)} \|f\|^2_{L^{2}(X)},
  \end{eqnarray}

 Thus, by the equation \eqref{concentration12}
 $$ \|\mathcal{G}_{\varphi}\{f\}\chi_{\Sigma}\|^{2}_{L^2(a^{\ast}\times B\times X)}\geq(1-m(\Sigma)^2) \|\varphi\|^2_{L^{2}(X)} \|f\|^2_{L^{2}(X)}$$
 $$\displaystyle\|f\|_{L^2(X)} \|\varphi\|_{L^{2}(X)}\leq \frac{1}{\sqrt{1-m(\Sigma)^2}}\|\mathcal{G}_{\varphi}\{f\}\{f\}\chi_{\Sigma^{c}}\|_{L^2(a^{\ast}\times B\times X)}$$

\end{proof}
\begin{theorem}[Concentration of $\mathcal{G}_{\varphi}\{f\}$ in small sets  ]\leavevmode\par
Let $\varphi$ be a window function and $\Sigma\subset a^{\ast}\times B\times X$ with $m(\Sigma)<1$.\\ Then, for $f\in L^2(X)$ we have
\begin{equation}\label{concentration-small}
\|\mathcal{G}_{\varphi}\{f\}-\chi_{\Sigma} \mathcal{G}_{\varphi}\{f\}\|_{L^2(a^{\ast}\times B\times X)}\geq \|\varphi\|_{L^2(X)}\|f\|_{L^2(X)}(1-m(\Sigma)\|_{L^2(X)}
\end{equation}

\end{theorem}

\begin{proof}

We have
$$\|\mathcal{G}_{\varphi}\{f\}-\chi_{\Sigma} \mathcal{G}_{\varphi}\{f\}\|_{L^2(a^{\ast}\times B\times X)}=\|\mathcal{G}_{\varphi}\{f\}(1-\chi_{\Sigma}) \|_{L^2(a^{\ast}\times B\times X)}$$
From theorem \ref{Plancherel}
\begin{eqnarray*}
   \|\mathcal{G}_{\varphi}\{f\}-\chi_{\Sigma} \mathcal{G}_{\varphi}\{f\}\|_{L^2(a^{\ast}\times B\times X)}&\geq& \|\mathcal{G}_{\varphi}\{f\}\|_{L^2(a^{\ast}\times B\times X)}\left(1-m(\Sigma) \right)\\
   &\geq& \|\varphi\|_{L^2(X)}\|f\|_{L^2(X)}\left(1-m(\Sigma) \right)
\end{eqnarray*}
Hence,
$$\|\mathcal{G}_{\varphi}\{f\}-\chi_{\Sigma} \mathcal{G}_{\varphi}\{f\}\|_{L^2(a^{\ast}\times B\times X)}\geq \|\varphi\|_{L^2(X)}\|f\|_{L^2(X)}\left(1-m(\Sigma)\right)$$

\end{proof}

\begin{theorem}
  Let $s>0$. Then there exists a constant $C_{s}>0$ such that, for all $f,\varphi\in L^{2}(X)$
  \begin{equation}\label{concentration2}
   \|f\|_{L^{2}(X)}\|\varphi\|_{L^{2}(X)} \leq C_{s}\left(\int_{a^{\ast}\times B\times X}
    {|(\lambda,b,h)|}^{2s}|\mathcal{G}_{\varphi}\{f\}(\lambda,b,h)|^2|c(\lambda)|^{-2}d\lambda db dh\right)^{\frac{1}{2}}
  \end{equation}
\end{theorem}

\begin{proof}
  Let $0<r\leq1$ be a real number and $B_{r}=\{(\lambda,b,h)\in a^{\ast}\times B\times X: |(\lambda,b,h)|<r\}$ the ball of center 0 and radius $r$ in $a^{\ast}\times B\times X$. Fix $0<t_{0}\leq1$ small enough such that $m(B_{t_{0}})<1$.\\
   Therefore, by inequality \eqref{theo-concentration} we obtain
  \begin{eqnarray*}
    \|f\|^2_{L^2(X)}\|\varphi\|^2_{L^2(X)}
     &\leq&  \frac{1}{t_{0}^{2s}(1-m(B_{t_{0}}))}\int_{|(\lambda,b,h)|>t_{0}} t_{0}^{2s}|\mathcal{G}_{\varphi}\{f\}(\lambda,b,h)|^2|c(\lambda)|^{-2}d\lambda db dh \\
     &\leq&  \frac{1}{t_{0}^{2s}(1-m(B_{t_{0}}))}\int\int_{|(\lambda,b,h)|>t_{0}} {|(\lambda,b,h)|}^{2s}|\mathcal{G}_{\varphi}\{f\}(\lambda,b,h)|^2|c(\lambda)|^{-2}d\lambda db dh\\
     &\leq&\frac{1}{t_{0}^{2s}(1-m(B_{t_{0}}))}\int\int_{a^{\ast}\times B\times X}{|(\lambda,b,h)|}^{2s}|\mathcal{G}_{\varphi}\{f\}(\lambda,b,h)|^2
    |c(\lambda)|^{-2}d\lambda db dh
  \end{eqnarray*}
  then,
  \begin{equation}\label{equa5}
    \|f\|^2_{L^2(X)}\|\varphi\|^2_{L^2(X)}\leq \frac{1}{t_{0}^{2s}(1-m(B_{t_{0}}))}\int\int_{a^{\ast}\times B\times X}{|(\lambda,b,h)|}^{2s}|\mathcal{G}_{\varphi}\{f\}(\lambda,b,h)|^2
    |c(\lambda)|^{-2}d\lambda db dh
  \end{equation}
  we take the square of both sides of inequality \eqref{equa5}, we get,
  $$\|f\|_{L^2(X)}\|\varphi\|_{L^2(X)}\leq \frac{1}{t_{0}^{s}\sqrt{1-m(B_{t_{0}})}}\left(\int\int_{a^{\ast}\times B\times X}{|(\lambda,b,h)|}^{2s}|\mathcal{G}_{\varphi}\{f\}(\lambda,b,h)|^2
    |c(\lambda)|^{-2}d\lambda db dh\right)^{\frac{1}{2}}$$
    We obtain the desired result by taking $C_{s}=t_{0}^{s}\sqrt{1-m(B_{t_{0}})}$.
\end{proof}

Now, before to give a version of Benedicks-type theorem for the Gabor Helgason Fourier transform, we start by giving the following notations and  definition.\\
Let $P_{\varphi}:L^2(X)\rightarrow L^2(X)$ be orthogonal  projection of $L^2(X)$ on the space $G_{\varphi}(L^2(X))$.\\
Let $\Sigma$ be a measurable subset of $a^{\ast}\times B\times X$ such that $0<m(\Sigma)<+\infty$, where $m$ is the Haar measure of $a^{\ast}\times B\times X$ , we consider the operator $P_{\Sigma}$ defined on $L^{2}(X)$ by $P_{\Sigma}F=\chi_{\Sigma}F$ where $F\in L^{2}(a^{\ast}\times B\times X)$ \\
The usual norm of the operator $P_{\Sigma}P_{\varphi}$ is defined by
$$\|P_{\Sigma}P_{\varphi}\|=sup\{\|P_{\Sigma}P_{\varphi}(F)\|_{L^2(a^{\ast}\times B\times X)}, \|F\|_{L^2(X)}\leq1\}.$$
\begin{definition}
  Let $\Sigma$ a measurable subset of $a^{\ast}\times B\times G$ and $\varphi\in L^{2}(X)$ a nonzero window function. We say that $\Sigma$ is weakly annihilating, if any function $f\in L^{2}(X)$ vanishes when its Helgason Gabor Fourier transform with respect to the window $\varphi$ is supported in $\Sigma$.
\end{definition}

\begin{theorem}[Benedicks-type uncertainty principle for $G_\varphi$]\leavevmode\par
 Let $r,R>0$. Let $\varphi\in L^2(X)\cap L^{\infty}(X)$ be a non zero window function such that $supp\varphi\subset B_r$ and let $\Sigma=S\times B_R\subset a^\ast\times B \times X$, be a subset of finite measure. Then
 $$Im \{P_{\varphi}\}\cap Im\{P_{\Sigma}\}=\{0\}$$
 i.e, $\Sigma$ is weakly annihilating.
\end{theorem}

\begin{proof}
  Let $F \in Im \{P_{\varphi}\}\cap Im\{P_{\Sigma}\}$, then, there exists a function $f\in L^2(X)$ such that, $F=G_{\varphi}(f)$ and $supp\{F\}\subset\Sigma$.\\
  Then for all $(\lambda,h,b)\in \Sigma$
  $$F(\lambda,h,b)=\widehat{(fT_{h^{-1}}\varphi)}(\lambda,b)$$
  Thus $supp\{\widehat{(fT_{h^{-1}}\varphi)}\}\subset S $, with $m(S)<+\infty$. On other hand $supp\varphi\subset B_r$, we have $supp\{fT_{h^{-1}}\varphi\}\subset B_{r+R}$\\
  Hence, by the Benedicks theorem of the Helgason transform(see theorem 6.1 in \cite{Mohanty}).\\
  We deduce that $fT_{h^{-1}}\varphi\equiv 0$ then $F=0$.
\end{proof}

\begin{center}

\end{center}

\end{document}